\documentclass[a4paper,11pt]{amsart}
\usepackage[utf8x]{inputenc}
\usepackage{amsmath,amssymb,amsthm}
\usepackage{graphics,color, graphicx}
\usepackage{colonequals}
\usepackage{textcomp}
\usepackage{url,hyperref}

\newtheorem{dfn}{Definition}
\newtheorem{thm}[dfn]{Theorem}
\newtheorem{prp}[dfn]{Proposition}

\newtheorem{lem}[dfn]{Lemma}
\theoremstyle{remark}
\newtheorem{rem}[dfn]{Remark}

\def\B{\mathbb{B}}
\def\inn{\operatorname{int}}
\def\R{\mathbb{R}}
\def\darea{\mathrm{darea}}

\def\arcosh{\operatorname{arcosh}}

\title[Isoperimetric inequality on cone-surfaces]{A simple proof\\ of an isoperimetric inequality\\ for euclidean and hyperbolic cone-surfaces}
\author{Ivan Izmestiev}
\thanks{Supported by the European Research Council under the European Union's Seventh Framework Programme (FP7/2007-2013)/\allowbreak ERC Grant agreement no.~247029-SDModels}
\address{Institut f\"ur Mathematik \\
Freie Universit\"at Berlin \\
Arnimallee 2 \\
D-14195 Berlin \\
 GERMANY}
\email{izmestiev@math.fu-berlin.de}

\begin{document}

\begin{abstract}
We prove that the isoperimetric inequalities in the euclidean and hyperbolic plane hold for all euclidean, respectively hyperbolic, cone-metrics on a disk with singularities of negative curvature. This is a discrete analog of the theorems of Weil and Bol that deal with Riemannian metrics of curvature bounded from above by $0$, respectively by $-1$. A stronger discrete version was proved by A.~D.~Alexandrov, with a subsequent extension by approximation to metrics of bounded integral curvature.

Our proof uses ``discrete conformal deformations'' of the metric that eliminate the singularities and increase the area. Therefore it resembles Weil's argument that uses the uniformization theorem and the harmonic minorant of a subharmonic function.
\end{abstract}

\maketitle

\section{Introduction}
\subsection{The main theorem}
A \emph{euclidean cone-metric} $g$ on a closed surface $M$ is a path metric structure such that every point has a neighborhood isometric either to an open euclidean disk or to a neighborhood of the apex of a euclidean cone with angle $\omega \in (0, +\infty) \setminus \{2\pi\}$ around the apex.
If $M$ has non-empty boundary, then we require that every boundary point has a neighborhood isometric either to a half-disk or to a circular sector of angle $\theta \in (0, +\infty) \setminus \{\pi\}$. \emph{Hyperbolic cone-metrics} on surfaces are defined similarly. A typical example is the metric space obtained by gluing together euclidean (respectively hyperbolic) triangles. Conversely, every cone-surface can be triangulated so that the metric induced on the triangles is euclidean, respectively hyperbolic.

The set of cone-like interior and angle-like boundary points is called the \emph{singular locus} of the metric $g$. An interior cone point with angle $\omega$ is said to have \emph{curvature} $2\pi - \omega$.

\begin{thm}
\label{thm:Main}
For every euclidean cone-metric $g$ on a disk $\B^2$ such that all cone points have negative curvatures the following inequality holds:
\begin{equation}
\label{eqn:IsoperEucl}
L^2 \ge 4\pi A
\end{equation}
where $A$ is the area and $L$ the perimeter of $(\B^2, g)$.

For every hyperbolic cone-metric $g$ on a disk $\B^2$ such that all cone points have negative curvatures the following inequality holds:
\begin{equation}
\label{eqn:IsoperHyp}
L^2 \ge 4\pi A + A^2
\end{equation}
\end{thm}

Inequalities \eqref{eqn:IsoperEucl}, respectively \eqref{eqn:IsoperHyp} hold for all euclidean, respectively hyperbolic metrics on a disk, as a consequence of the isoperimetric inequalities in the euclidean, respectively hyperbolic, plane. Therefore Theorem \ref{thm:Main} is implied by the following proposition.

\begin{prp}
\label{prp:ElimCone}
For every euclidean or hyperbolic cone-metric $g$ on $\B^2$ such that all cone points have negative curvatures there is a euclidean, respectively hyperbolic, metric on $\B^2$ with the same perimeter and larger area.
\end{prp}

Stronger versions of Proposition \ref{prp:ElimCone} and Theorem \ref{thm:Main} were proved by A.~D.~Ale\-xandrov, see Section \ref{sec:Alex} below. The aim of the present article is to give a new proof that is simple and in some sense conceptually attractive.

\subsection{Weil's isoperimetric problem}
Theorem \ref{thm:Main} can be viewed as the discrete analog of the following theorem.

\begin{thm}
\label{thm:Weil}
For every Riemannian metric on a disk $\B^2$ with the Gauss curvature $K(x) \le 0$ the euclidean isoperimetric inequality holds.

For every Riemannian metric on a disk $\B^2$ with $K(x) \le -1$ the hyperbolic isoperimetric inequality holds.
\end{thm}

The first part was proved independently by Weil \cite{Weil26} and Beckenbach and Rad\'o \cite{BR33}. The second part is due to Bol \cite{Bol41}.

Aubin \cite{Aub76} conjectured that a similar result holds in higher dimensions: a simply connected $n$-manifold of non-positive sectional curvature satisfies the isoperimetric inequality of $\R^n$. Later the conjecture was extended to metrics with sectional curvature bounded by $\kappa \le 0$, replacing $\R^n$ by the space-form of curvature $\kappa$. As for now, only the cases $n=3$ for any $\kappa \le 0$ \cite{Kle92} and $n=4$ for $\kappa = 0$ \cite{Cro84} are verified. See \cite{KK13+} for a novel approach and new partial results.

\subsection{Surfaces of bounded curvature in the sense of Alexandrov}
\label{sec:Alex}
A.~D.~Alexandrov's stronger version of Theorem \ref{thm:Main} is
\begin{equation}
\label{eqn:AlexStrong}
L^2 \ge 2(2\pi - \kappa^+)A - kA^2
\end{equation}
where the cone-surface is allowed to have singularities of positive curvature, $\kappa^+$ denotes the sum of all positive curvatures, and $k$ stands for the curvature of the model space ($k=0$ in the euclidean and $k=-1$ in the hyperbolic case). Alexandrov's method consists in repeated cutting of the surface along piecewise geodesics paths and gluing in polygonal regions along the cuts.

By approximation, inequality \eqref{eqn:AlexStrong} (with $\kappa^+$ duly defined in dependence on $k$) holds for all metrics of \emph{bounded integral curvature}, a broad class introduced by Alexandrov, that includes both Riemannian and cone-metrics. The equality holds only if $M$ is a circular neighborhood of the apex of a cone. See \cite[Section 2.2]{BZ88} for more details and references. In particular, this generalizes the inequality
\[
L^2 \ge 2 \left( 2\pi - \int_M K\, \darea \right)A
\]
proved first by Fiala \cite{Fia41} for analytic Riemannian metrics of positive Gauss curvature.


\subsection{Subharmonic functions and conformal deformations}
Weil's proof of the first part of Theorem \ref{thm:Weil} goes as follows. A metric with non-positive Gauss curvature uniformizes with a subharmonic conformal factor:
\[
g = e^{-2u} \tilde{g}, \quad 0 \ge K = e^{2u} \Delta u
\]
where the metric $\tilde{g}$ is flat. A harmonic function that coincides with $u$ on $\partial M$ minorizes $u$:
\[
\Delta v = 0, \, v|_{\partial M} = u|_{\partial M} \Rightarrow v \le u
\]
Hence for a flat metric $g' \colonequals e^{-2v}\tilde{g}$ we have
\[
L' = L, \quad A' = \int_M e^{-2v}\, \darea \ge \int_M e^{-2u}\, \darea = A
\]
where the integration is done with respect to the area element of $\tilde{g}$. Thus the theorem is reduced to the isoperimetric problem in the euclidean plane (if one is able to deal with possible self-overlaps of the development of $(M,g')$ onto $\R^2$).

\subsection{Plan of the paper}
In Section \ref{sec:Triang} we show that every euclidean or hyperbolic cone-surface (independently of its topological type and curvature signs) can be geodesically triangulated without adding unnecessary vertices. Namely, all interior vertices of the triangulation are cone-points, and one is allowed to add a vertex on a boundary component if it is geodesic, that is contains no angle points.

In Section \ref{sec:Deform} we explain how to deform the edge lengths in such a minimal triangulation of a negatively curved cone-surface so that its area increases and one of the interior singularities disappears. By iterating this step, taking each time a new triangulation with fewer vertices, we arrive to a metric without cone points that has a larger area and the same perimeter as the initial cone-metric. This proves Proposition \ref{prp:ElimCone} and hence Theorem \ref{thm:Main}.

The basic deformation step is a special elongation of all edges incident to some interior vertex. This can be viewed as a discrete analog of a conformal deformation with non-negative factor. In this respect our method resembles Weil's argument from the previous section.

Note that a different discrete analog of conformal deformations was proposed in \cite{Luo04}.

\section{Triangulating cone-surfaces}
\label{sec:Triang}

Let $(M,g)$ be a euclidean or hyperbolic cone-surface, possibly with boundary. A \emph{geodesic triangulation} of $(M,g)$ is a decomposition of $M$ into euclidean, respectively hyperbolic, triangles with disjoint interiors such that every side of every triangle is either contained in the boundary of $M$ or glued to another side of another or the same triangle. Clearly, the vertex set of a geodesic triangulation contains the singular locus of $g$.

Every cone-surface can be geodesically triangulated.
The main result of this section is Proposition \ref{prp:TriangSurf}, which implies the existence of a \emph{minimal} triangulation, that is one all of whose vertices are singular points of the metric, with addition of one extra vertex on every geodesic component of the boundary. (If $M$ is closed and has no cone points, then one has to add a non-singular vertex in the interior.) The result is probably not new, but we couldn't find a reference. An analogous statement is false in dimension~$3$, see Remark \ref{rem:Triang3}.

By a \emph{scissors cut} we mean a simple geodesic that starts at a point $p \in V \cap \partial M$ and ends at a point $q \in V$ without meeting neither points of $V$ nor $\partial M$ on the way.

\begin{prp}
\label{prp:TriangSurf}
Let $(M,g)$ be a euclidean or hyperbolic cone-surface, possibly with boundary, and $V \subset M$ be a finite non-empty set, such that
\begin{itemize}
\item $V$ contains the singular locus of $g$;
\item every component of $\partial M$ contains at least one point from $V$.
\end{itemize}
Then there exists a geodesic triangulation of $(M,g)$ with the vertex set $V$.
\end{prp}
\begin{proof}
If $\partial M = \emptyset$, then create a boundary by cutting $M$ along any simple geodesic arc with the endpoints, and only them, in $V$. To show that such a geodesic exists, one can use the exponential map based at a point $p \in V$.


Thus we may assume $\partial M \ne \emptyset$. By Lemma \ref{lem:ConcAngle}, if $M$ is not a triangle with vertex set $V$, then a scissors cut exists. It yields a new one or two cone-surfaces with boundary. Continue cutting until all components become triangles. In order to show that this will happen, use induction with respect to the lexicographic order on the set of pairs $(-\chi, 3m+n)$, where $\chi$ is the Euler characteristic, $m = |V \cap \inn M|$, and $n = |V \cap \partial M|$. That is, we will show that if $M'$ is a component obtained from $M$ by a scissors cut, then either $\chi(M') > \chi(M)$ or $\chi(M') = \chi(M)$ and $3m'+n'< 3m+n$.

Assume that the other end of the cut belongs to $\partial M$. If cutting yields two components $M_1$ and $M_2$, then either $\chi(M_i) < \chi(M)$ for $i = 1, 2$, so that $M_i < M$ in the lexicographic order described above, or without loss of generality $\chi(M_1) = \chi(M)$ and $M_2 \approx \B^2$. In the latter case $M_2$ contains a singular point different from the endpoints of the cut. Therefore $3m_1 + n_1 < 3m + n$ and $M_1 < M$. To show that $M_2 < M$, consider two cases. If $\chi(M) < 1$, then $\chi(M_2) > \chi(M)$; if $\chi(M) = 1$, then $M_1$ contains a singular point  different from the endpoints of the cut, so that $3m_2 + n_2 < 3m + n$. In both cases we have $M_2 < M$.

Now assume that the other end of the cut lies in the interior of $M$. Then cutting along $\gamma$ we obtain a cone-surface $M'$ with $\chi(M') = \chi(M)$ and $m' = m - 1$, $n' = n + 2$, so that $3m'+n'=3m+n-1$ and hence $M' < M$.
\end{proof}

\begin{lem}
\label{lem:ConcAngle}
Let $(M,g)$ and $V$ be as in Proposition \ref{prp:TriangSurf}, and let $C$ be a component of $\partial M$.

If for all $p \in V \cap C$ the angle at $p$ is less than $\pi$, and $(M,g)$ is not a triangle with the vertex set $V$, then there is a scissors cut starting from every $p \in C \cap V$.

If for some $p \in V \cap C$ the angle at $p$ is bigger or equal $\pi$, then there is a scissors cut starting at this $p$.
\end{lem}
\begin{proof}
Assume that the angles at all boundary vertices are less than $\pi$. If $V \subset \partial M$, then $M$ is isometric to a euclidean or hyperbolic convex polygon with the vertex set $V$, so that $|V| > 3$ implies that there is a diagonal starting at any $p \in V$. If there is $q \in V \cap \inn M$, then take the shortest path $\gamma$ from $p$ to $q$. Due to the convexity of the boundary, $\gamma \cap \partial M = \{p\}$. If $\gamma \cap V = \{p,q\}$, then $\gamma$ is a scissors cut, otherwise stop cutting at the point on $\gamma \cap V$ which is the closest to $p$.

Now assume that $p \in V \cap C$ is such that the angle at $p$ is at least $\pi$. Consider the exponential map $\exp_p$, restricted to the interior of the angle at $p$. If at some radius $r$ it ceases to be injective, then we either find a simple interior geodesic of length $r$ ending at a cone point of positive curvature, or we find a simple closed geodesic of length $2r$ based at $p$. If at some radius $r$ the exponential map meets an edge $e$ at its interior point $q$, then move $q$ along $e$ and look what happens with the geodesic $pq$. It will either meet a point from $V$, or $q$ arrives an endpoint of $e$, or its initial segment will meet the boundary. The latter cannot happen in both directions along $e$, since the angle at $p$ is at least $\pi$. The former two possibilities with the excluded latter yield a scissors cut starting at $p$.
\end{proof}

\begin{rem}
For spherical cone-surfaces a geodesic triangulation without additional points does not always exist. A necessary assumption is that the surface does not contain a subset isometric to an open hemisphere.
\end{rem}

\begin{rem}
\label{rem:Triang3}
Lemma \ref{lem:ConcAngle} is well-known for euclidean polygons and is used to prove that every non-convex polygon can be triangulated without additional vertices, \cite{Len11}. In the same article an example was given of a non-convex $3$-dimensional polytope that cannot be triangulated without additional vertices. A simpler example was given by Sch\"onhardt (a twisted octahedron).

Sch\"onhardt and Lennes polyhedra provide examples of cone-manifolds that cannot be triangulated without additional vertices. By filling the concavities of the Sch\"onhardt octahedron by tetrahedra so that three singular interior edges are created one obtains a non-triangulable example with convex boundary. It seems that the double of the Sch\"onhardt's octahedron also cannot be triangulated without additional vertices.
\end{rem}

\section{Deforming a metric}
\label{sec:Deform}
\begin{prp}
Let $(M,g)$ be a euclidean or hyperbolic cone-surface with at least one cone point and negative curvatures at all cone points. Then there is a cone-metric $g'$ on $M$ with a smaller number of cone-points and
\[
L' = L, \quad A' > A
\]
where $L$ and $A$ are the total length of boundary components and the area of $M$ with respect to the metric $g$, and $L'$ and $A'$ are the corresponding values for $g'$.
\end{prp}
\begin{proof}
Choose a geodesic triangulation of $(M,g)$ such that all of its interior vertices are cone-points of $g$. This is possible due to Proposition~\ref{prp:TriangSurf}. Pick an interior vertex $p$ and consider a family of cone-metrics $g_t$ that coincide with $g$ outside of the star of $p$ and where every edge $pq$ is deformed according to
\[
\ell_{pq}(t) =
\begin{cases}
\sqrt{\ell_{pq}^2 + t}, &\text{ in the euclidean case}\\
\arcosh(e^t \cosh b), &\text{ in the hyperbolic case}
\end{cases}
\]
Here $\ell_{pq}$ is the length of $pq$ in the metric $g$. By Lemma \ref{lem:NewLengths}, the new edge lengths satisfy the triangle inequalities, therefore can be used to replace the triangles in the star of $p$ with new triangles. Let $t_0$ be the minimum $t$ for which the cone angle at $p$ or at one of its neighbors becomes equal to $2\pi$. Such a $t$ exists, because by Lemma \ref{lem:NewLengths} the angle at $p$ tends to $0$ as $t$ tends to $+\infty$. By Lemma \ref{lem:NewLengths}, the metric $g' = g_{t_0}$ has a larger area than $g$.
\end{proof}

\begin{lem}
\label{lem:NewLengths}
Let $\Delta$ be a euclidean or hyperbolic triangle with side lengths $a$, $b$, $c$. Then for every $t > 0$ the triangle $\Delta_t$ with side lengths
\[
\begin{aligned}
&a, \sqrt{b^2 + t}, \sqrt{c^2 + t} \quad &\text{in the euclidean case}\\
&a, \arcosh(e^t \cosh b), \arcosh(e^t \cosh c) &\text{in the hyperbolic case}
\end{aligned}
\]
exists and has a larger area than the triangle $\Delta$. Besides, the angle of $\Delta_t$ opposite to the side $a$ tends to $0$ as $t$ tends to $+\infty$.
\end{lem}
\begin{proof}
The Pythagorean theorem for euclidean and hyperbolic right-angled triangles implies that the triangle $\Delta_t$ can be obtained from the triangle $\Delta$ by moving the vertex $A$ opposite to the side $a$ along the perpendicular to this side. Equivalently, $\Delta_t$ is a side of a triangular pyramid with the base $\Delta$ and the apex directly over the vertex $A$, see Figure \ref{fig:Lengths}.

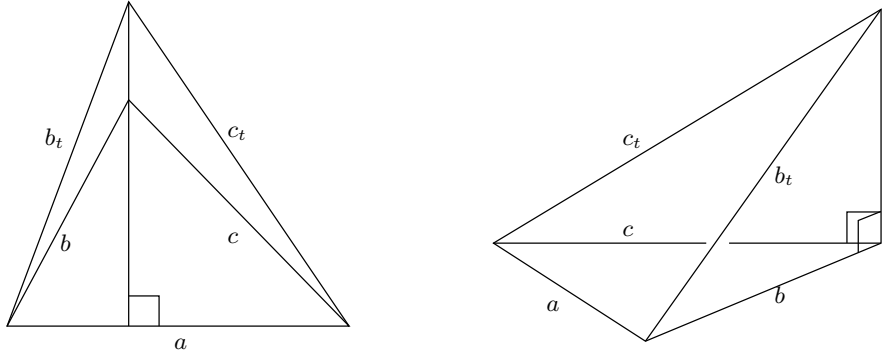
\begin{figure}[ht]
\begin{center}
\begin{picture}(0,0)%
\includegraphics{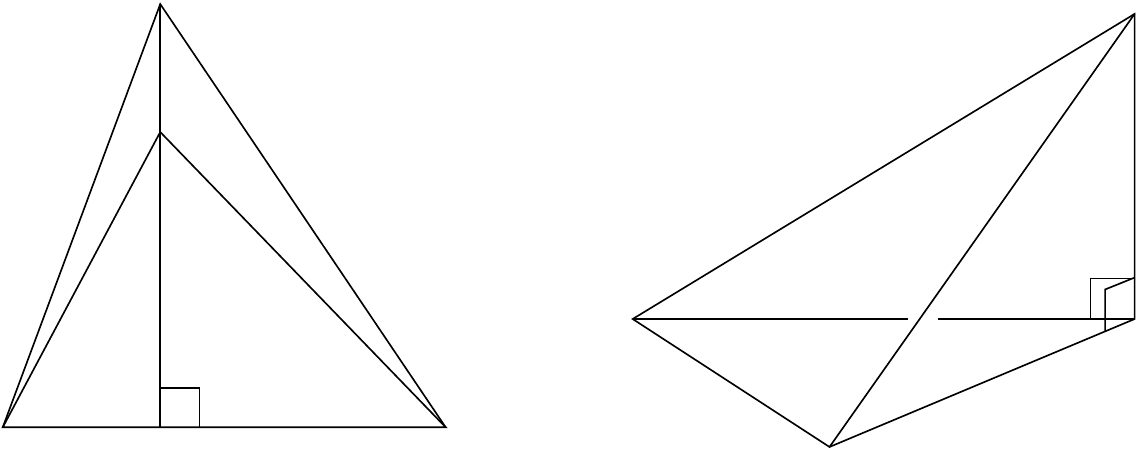}%
\end{picture}%
\setlength{\unitlength}{4144sp}%
\begingroup\makeatletter\ifx\SetFigFont\undefined%
\gdef\SetFigFont#1#2#3#4#5{%
  \reset@font\fontsize{#1}{#2pt}%
  \fontfamily{#3}\fontseries{#4}\fontshape{#5}%
  \selectfont}%
\fi\endgroup%
\begin{picture}(5199,2137)(-11,-1376)
\put(3196,-1096){\makebox(0,0)[lb]{\smash{{\SetFigFont{9}{10.8}{\rmdefault}{\mddefault}{\updefault}{\color[rgb]{0,0,0}$a$}%
}}}}
\put(4546,-1051){\makebox(0,0)[lb]{\smash{{\SetFigFont{9}{10.8}{\rmdefault}{\mddefault}{\updefault}{\color[rgb]{0,0,0}$b$}%
}}}}
\put(4546,-331){\makebox(0,0)[lb]{\smash{{\SetFigFont{9}{10.8}{\rmdefault}{\mddefault}{\updefault}{\color[rgb]{0,0,0}$b_t$}%
}}}}
\put(3646,-646){\makebox(0,0)[lb]{\smash{{\SetFigFont{9}{10.8}{\rmdefault}{\mddefault}{\updefault}{\color[rgb]{0,0,0}$c$}%
}}}}
\put(3646,-106){\makebox(0,0)[lb]{\smash{{\SetFigFont{9}{10.8}{\rmdefault}{\mddefault}{\updefault}{\color[rgb]{0,0,0}$c_t$}%
}}}}
\put(316,-736){\makebox(0,0)[lb]{\smash{{\SetFigFont{9}{10.8}{\rmdefault}{\mddefault}{\updefault}{\color[rgb]{0,0,0}$b$}%
}}}}
\put(226,-106){\makebox(0,0)[lb]{\smash{{\SetFigFont{9}{10.8}{\rmdefault}{\mddefault}{\updefault}{\color[rgb]{0,0,0}$b_t$}%
}}}}
\put(991,-1321){\makebox(0,0)[lb]{\smash{{\SetFigFont{9}{10.8}{\rmdefault}{\mddefault}{\updefault}{\color[rgb]{0,0,0}$a$}%
}}}}
\put(1306,-691){\makebox(0,0)[lb]{\smash{{\SetFigFont{9}{10.8}{\rmdefault}{\mddefault}{\updefault}{\color[rgb]{0,0,0}$c$}%
}}}}
\put(1306,-61){\makebox(0,0)[lb]{\smash{{\SetFigFont{9}{10.8}{\rmdefault}{\mddefault}{\updefault}{\color[rgb]{0,0,0}$c_t$}%
}}}}
\end{picture}%

\end{center}
\caption{Deformation of edge lengths.}
\label{fig:Lengths}
\end{figure}

The first realization of $\Delta_t$ implies that the angle opposite to the side $a$ tends to $0$ as the vertex goes to infinity. The second realization implies that $\Delta_t$ has a larger area than $\Delta$: the orthogonal projection decreases the areas both in the euclidean and in the hyperbolic space. In the euclidean geometry the lengths parallel to the side $a$ are preserved by the projection, while those orthogonal to $a$ are decreased; in the hyperbolic geometry the lengths in both directions are increased, as follows from consideration of quadrilaterals with two adjacent right angles.
\end{proof}


\end{document}